\documentclass[10pt]{amsart}

\usepackage{amssymb,latexsym, mathtools,tikz}
\usepackage{enumerate}
\usepackage{graphicx}
\usepackage{float}
\usepackage{placeins}
\usepackage{mdframed}
\usepackage{amssymb}
\usepackage{esint}
\usepackage{cool}
\usepackage[all,cmtip]{xy}
\usepackage{mathtools}
\usepackage{amstext} 
\usepackage{array}   
\usepackage[shortlabels]{enumitem}
\usepackage{ytableau}

\newcolumntype{L}{>{$}l<{$}} 
\newcolumntype{C}{>{$}c<{$}}
\newtheorem{theorem}{Theorem}[section]
\newtheorem{lemma}[theorem]{Lemma}
\newtheorem{cor}[theorem]{Corollary}
\newtheorem{prop}[theorem]{Proposition}
\newtheorem{setup}[theorem]{Setup}
\theoremstyle{definition}
\newtheorem{definition}[theorem]{Definition}
\newtheorem{example}[theorem]{Example}

\newtheorem{obs}[theorem]{Observation}
\newtheorem{notation}[theorem]{Notation}

\theoremstyle{remark}
\newtheorem{remark}[theorem]{Remark}

\newtheorem{the context}[theorem]{The Context}
\newtheorem{question}[theorem]{Question}
\numberwithin{equation}{theorem}
\numberwithin{equation}{section}

\newcommand{\cat}[1]{\mathcal{#1}}

\newcommand{\rank}{\operatorname{rank}}

\newcommand{\grade}{\operatorname{grade}}

\newcommand{\im}{\operatorname{Im}}

\newcommand{\Ker}{\operatorname{Ker}}

\newcommand{\ass}{\operatorname{Ass}}

\newcommand{\supp}{\operatorname{Supp}}

\newcommand{\bbz}{\mathbb{Z}}

\newcommand{\bbp}{\mathbb{P}}

\renewcommand{\geq}{\geqslant}
\renewcommand{\leq}{\leqslant}
\renewcommand{\ker}{\Ker}
\renewcommand{\hom}{\Hom}

\newcommand{\Hom}{\operatorname{Hom}}

\newcommand{\maps}[5]{\xymatrix{#1 \ar[r]^-{#3} & #2 \\
#4 \ar@{|->}[r] & #5 \\}}

\newcommand{\mfa}{\mathfrak{a}}

\newcommand{\lcm}{\textrm{lcm}}

\def\w{\wedge}

\def\im{\operatorname{im}}

\newcommand{\sgn}{\operatorname{sgn}}
\newcommand{\po}{\operatorname{po}}
\newcommand{\cX}{\mathcal{X}}
\newcommand{\cP}{\mathcal{P}}
\newcommand{\ZZ}{\mathbb{Z}}
\newcommand{\clique}{\Delta^{\textrm{clique}}}
\renewcommand{\aa}{\textbf{a}}
\newcommand{\bb}{\textbf{b}}
\newcommand{\lin}{\operatorname{lin}}
\newcommand{\rain}{\operatorname{rain}}
\newcommand{\mdeg}{\operatorname{mdeg}}

\newcommand{\inn}{\textrm{in}}

\begin{document}
\title{Linear Strands Supported on Regular CW Complexes}

\author{Keller VandeBogert }
\date{\today}

\maketitle

\begin{abstract}
In this paper, we study ideals $I$ whose linear strand can be supported on a regular CW complex. We provide a sufficient condition for the linear strand of an arbitrary subideal of $I$ to remain supported on an easily described subcomplex. In particular, we prove that a certain class of rainbow monomial ideals always have linear strand supported on a regular CW complex, including any initial ideal of the ideal of maximal minors of a generic matrix.  We also provide a sufficient condition for these ideals to have linear resolution, which is also an equivalence under mild assumptions. We then employ a result of Almousa, Fl\o ystad, and Lohne to apply these results to polarizations of Artinian monomial ideals. We conclude with further questions relating to cellularity of certain classes of squarefree monomial ideals and the relationship between initial ideals of maximal minors and algebra structures on certain resolutions.
\end{abstract}

\section{Introduction}

A common theme in the study of monomial ideals is the association of some kind of combinatorial object $P$ that can recover the ideal of interest $I$. It is then often the case that properties desirable from a combinatorial perspective for $P$ translate into properties that are desirable from an algebraic perspective for $I$. This kind of correspondence has been used to great effect in, for instance, Stanley-Reisner theory, where any squarefree monomial ideal has a bijective correspondence with a simplicial complex. This correspondence yields nontrivial consequences, such as Hochster's formula or Reisner's criterion for Cohen-Macaulayness.

From a homological perspective, the object that best captures the behavior of a monomial ideal $I$ is a free resolution. Thus, a similar line of thinking suggests that the association of a free resolution $F_\bullet$ of $I$ to some combinatorial object $P$ will yield a dictionary for translating properties of $P$ into properties of $F_\bullet$. This was first explored by Bayer, Peeva, and Sturmfels \cite{bayer1998monomial}, where free resolutions that could be supported on simplicial complexes were studied. This means that the differential of $F_\bullet$ and the natural differential on the simplicial complex induce each other, and there is a bijective correspondence between basis elements in a given homological degree and faces of a given dimension. Examples of well-known complexes supported on simplicial complexes include the Taylor resolution and the Scarf complex. However, simplicial complexes turn out to be too restrictive in general to support more than the most well-behaved complexes.

This leads to the idea of considering \emph{cellular or CW resolutions}, that is, resolutions supported on so-called cell/regular CW complexes, introduced by Bayer and Sturmfels in \cite{bayer1998cellular}. Cellular and CW resolutions have turned out to occupy a ``goldilocks" zone of generality; they are general enough to support minimal free resolutions of many large classes of monomial ideals, including the well-known Eliahou-Kervaire resolution (see \cite{mermin2010eliaiiou}), but not so general as to be combinatorially incomprehensible. There are many classes of ideals whose minimal free resolutions can be supported on cell or regular CW complexes; see, for instance \cite{craw2012cellular}, \cite{batzies2002discrete}, \cite{floystad2009cellular}, \cite{dochtermann2012cellular}, \cite{dochtermann2012tropical}, \cite{dochtermann2014cellular}, or \cite{develin2007tropical}. 

It turns out that not \emph{all} monomial ideals have cellular resolutions; indeed, there are monomial ideals whose minimal free resolution cannot even be supported on a CW-complex (this is due to Velasco in \cite{velasco2008minimal}). Reiner and Welker (see \cite{reiner2001linear}) have also constructed a relatively simple monomial ideal with linear resolution, but such that no change of basis has differentials with coefficients equal to $0$ or $\pm 1$. Minimal free resolutions of arbitrary monomial ideals are not completely devoid of structure, however; a result of Clark and Tchernev (see \cite{clark2019minimal}) shows that every monomial ideal has minimal free resolution that can be supported on a so-called hcw-poset. 

In \cite{NR09}, a specific cell complex called the \emph{complex of boxes} was used to support what is now called the \emph{box polarization} of squarefree strongly stable ideals. In general, polarizations are a method of associating to any monomial ideal $I$ a squarefree monomial $\widetilde{I}$ in a larger polynomial ring that is homologically indistinguishable. There are \emph{many} different ways to polarize a monomial ideal; indeed, Almousa, Fl\o ystad, and Lohne in \cite{polarizations} describe all such polarizations for powers of the graded maximal ideal in terms of spanning trees of certain maximal ``down triangles" associated to the graph of linear syzygies. Moreover, they also show that the study of polarizations of an Artinian monomial ideals is equivalent via Alexander duality to the study of ideals generated by \emph{rainbow monomials} with linear resolution. Rainbow monomial ideals also appear as defining ideals for sets of points in multiprojective spaces, where each color class corresponds to a distinct copy of $\bbp^n$; see the work by Favacchio, Guardo, and Migliore \cite{favacchio2018arithmetically} or \cite{favacchio2019multiprojective} for rainbow monomials from this perspective.

In this paper, we use these connections as motivation for studying a proper subclass of all equigenerated rainbow monomial ideals, which we call \emph{rainbow determinantal facet ideals} (or rainbow DFIs). These ideals are inspired by \emph{determinantal facet ideals} (DFIs), a generalization of binomial edge ideals (see \cite{DES98}, \cite{ohtani2011}, \cite{herzog2010binomial}, and \cite{madani16-BEIsurvey}); the main difference is that rainbow DFIs depend on both a term order $<$ \emph{and} a parametrizing simplicial complex $\Delta$. We prove that the linear strand of \emph{any} rainbow DFI is supported on a regular CW complex; in particular, any rainbow DFI with linear resolution (hence Alexander dual to a polarization) has a CW resolution. Moreover, we formulate a sufficient condition for any rainbow DFI to have linear resolution; if the Alexander dual $\Delta^\vee$ of the associated simplicial complex has nonmaximal overlap, this becomes an equivalence.  

The techniques involved in proving these statements are of independent interest and involve the formulation of more general statements on the structure of the minimal free resolution of any initial ideal of the ideal of maximal minors, the construction of linear strands of arbitrary modules, and CW-ness of complexes whose associated poset is CW. Moreover, we examine conditions guaranteeing that a sequence that is regular on $R/I$ is also regular on $R/J$ for any fixed $J \subset I$; this is used for computing the graded Betti numbers of rainbow DFIs (and is later used to give a characterization of Cohen-Macaulayness). 

The paper is organized as follows. In Section \ref{sec:backgroundEN}, we first introduce some necessary background and notation relating to (rainbow) DFIs, the Eagon-Northcott complex, and arbitrary initial ideals of ideals of maximal minors. We perform a more detailed study of so-called \emph{sparse} Eagon-Northcott complexes, including an explicit description of the differentials and the multidegrees appearing in each homological degree.

In Section \ref{sec:linStrands}, we consider the construction of linear strands in general. We first prove, using the machinery of \emph{iterated trimming complexes} (introduced in \cite{vandebogert2020trimming}), that the linear strand of any subideal with $G(J) \subseteq G(I)$ is obtained by taking an appropriate subcomplex of the linear strand of $I$. We then apply this result to the case that the linear strands in question are supported on cellular or CW complexes; we introduce the notion of \emph{support chains} centered at a vertex $v$ of a regular CW complex to construct a morphism of complexes. Taking the kernel of this morphism of complexes will yield the linear strand upon deleting of the generator corresponding to the vertex $v$. In the case that the CW complex is \emph{linearly connected}, the resulting linear strand is obtained by simply restricting to the subcomplex induced by deleting $v$. 

In Section \ref{sec:latticesAndCell}, we study conditions guaranteeing that certain posets associated to a complex $F_\bullet$ give rise to a regular CW complex supporting $F_\bullet$. Using results of Bj\"orner, CW-ness can be checked by showing that the associated poset is CW. We use this result to prove that any initial ideal of the ideal of maximal minors has CW resolution. Even more generally, we combine this with the results of Section \ref{sec:linStrands} to prove that the linear strand of every rainbow DFI is supported on a regular CW complex.

In Section \ref{sec:regandLinearity}, we study conditions on the simplicial complex $\Delta$ associated to a rainbow DFI implying linearity of the minimal free resolution. This condition is formulated in terms of the existence of an ordering on the elements of $\Delta^\vee$ such that they form a so-called \emph{free sequence}; moreover, if elements of $\Delta^\vee$ have nonmaximal overlap, then every ordering yields a free-sequence and this condition is also an equivalence. In this latter situation, we also compute the Betti numbers explicitly; we already know by the result of Section \ref{sec:latticesAndCell} that these ideals have CW resolution.

Finally, in Section \ref{sec:polarizationsandConc}, we apply our results to polarizations of Artinian monomial ideals. Taking Alexander duals, all of the conditions guaranteeing linearity of Section \ref{sec:regandLinearity} directly translate to the property of being the polarization of an Artinian monomial ideal. Moreover, assuming that all elements of $\Delta^\vee$ have nonmaximal overlap, we show that rainbow DFIs are Alexander dual to a polarization of a power of the maximal ideal if and only if $|\Delta^\vee| = 0$. We conclude with a variety of questions: we first ask about whether or not the admission of an order for which the elements of $\Delta^\vee$ form a free sequence is a necessary condition for the associated rainbow DFI to have linear resolution. We also ask about the existence of term orders for which a pre-selected set of generators are free vertices of the regular CW complex supporting the resolution of the initial ideal of maximal minors. Lastly, we pose a curious potential connection between the selection of a multigraded algebra structure on the minimal resolution of the ideal of all squarefree monomials, and an arbitrary initial ideal of an ideal of maximal minors.

\section{Background and Sparse Eagon-Northcott Complexes}\label{sec:backgroundEN}

In this section, we introduce some notation related to (rainbow) determinantal facet ideals that will be in play for the rest of the paper. We then delve into so-called \emph{sparse} Eagon-Northcott complexes, which arise by homogenizing the classical Eagon-Northcott complex with respect to a variable $t$ and term order $<$, then setting $t=0$. It turns out that one can say \emph{precisely} what the differentials and multidegrees must look like for any given term order. These complexes will be used in later sections for proving cellularity/CW-ness and establishing criteria for having linear resolution. 

Throughout this paper, all complexes will be assumed to be nontrivial only in nonnegative homological degrees. 

\begin{notation}
Let $R = k[x_{ij} \mid 1 \leq i \leq n, \ 1 \leq j \leq m]$ be a polynomial ring over an arbitrary field $k$. Let $M$ be an $n \times m$ matrix of variables in $R$ where $n\leq m$.
For indices $\aa = \{a_1,\ldots, a_r\}$ and $\bb = \{b_1, \ldots, b_r\}$ such that $1\leq a_1 < \ldots < a_r\leq n$ and $1\leq b_1 < \cdots < b_r\leq m$, set
$$
[\aa\vert\bb] = [a_1,\ldots, a_r\vert b_1,\ldots, b_r] = \det \left( \begin{array}{ccc}
    x_{a_1,b_1} & \cdots & x_{a_1,b_r} \\
    \vdots & \ddots & \vdots\\
    x_{a_r,b_1} & \cdots & x_{a_r,b_r}\\
\end{array} \right)
$$
where $[\aa\vert\bb]=0$ if $r>n$. When $r=n$, use the simplified notation $[\aa]$ = $[1,\ldots, n\vert \aa]$. The ideal generated by the $r$-minors of $M$ is denoted $I_r(M)$.
\end{notation}

\begin{definition}\label{def:cliques}
For a simplicial complex $\Delta$ and an integer $i$, the $i$-th skeleton $\Delta^{(i)}$ of $\Delta$ is the subcomplex of $\Delta$ whose faces are those faces of $\Delta$ with dimension at most $i$. Let $\cat{S}$ denote the set of simplices $\Gamma$ with vertices in $[m]$ with $\dim (\Gamma) \geq n-1$ and $\Gamma^{(n-1)} \subset \Delta$. 

Let $\Gamma_1 , \dotsc , \Gamma_c$ be maximal elements in $\cat{S}$ with respect to inclusion, and let $\Delta_i := \Gamma^{(n-1)}_i$. Each $\Gamma_i$ is called a \emph{maximal clique}, and any induced subcomplex of $\Gamma_i$ is a \emph{clique}. The simplicial complex $\clique$ whose facets are the maximal cliques of $\Delta$ is called the \emph{clique complex} associated to $\Delta$. The decomposition $\Delta = \Delta_1 \cup \cdots \cup \Delta_c$ is called the \emph{maximal clique decomposition} of $\Delta$.
\end{definition}

\begin{definition}\label{def: DFI}
Let $\Delta$ be a pure $(n-1)$-dimensional simplicial complex on the vertex set $[m]$. Let $R = k[x_{ij} \mid 1 \leq i \leq n, \ 1 \leq j \leq m]$ be a polynomial ring over an arbitrary field $k$. Let $M$ be an $n \times m$ matrix of variables in $R$. The \emph{determinantal facet ideal} (or \emph{DFI}) $J_\Delta\subseteq R$ associated to $\Delta$ is the ideal generated by determinants of the form $[\aa ]$ where $\aa$ supports an $(n-1)$-facet of $\Delta$; that is, the columns of $[\aa ]$ correspond to the vertices of some facet of $ \Delta$.

Let $<$ be an arbitrary term order. The \emph{rainbow DFI} $\rain_< (J_\Delta)$ associated to $\Delta$ is the monomial ideal generated by all monomials of the form $\inn_< ([\aa ])$, where $\aa$ supports and $(n-1)$-facet of $\Delta$. 
\end{definition}

Observe that the maximal clique decomposition of the associated simplicial complex $\Delta$ is telling us the largest submatrices of $M$ from which \emph{all} minors are taken. That is, $J_\Delta = J_{\Delta_1} + \cdots + J_{\Delta_r}$.

\begin{notation}
Let $\Delta$ be a pure $(n-1)$-dimensional simplicial complex on the vertex set $[m]$. The notation $\Delta^\vee$ will denote the Alexander dual of $\Delta$; that is, $\Delta^\vee$ is the unique $(n-1)$-pure simplicial complex with $\Delta \cup \Delta^\vee$ equal to all $n$-subsets of $[m]$.
\end{notation}

Recall that the set of maximal minors forms a universal Gr\"obner basis for $I_n (M)$ by a result of Sturmfels and Zelevinsky (see \cite{SturmfelsZelevinsky93}). This immediately yields the following:

\begin{obs}
Let $<$ be an arbitrary term order. For any $(n-1)$-pure simplicial complex $\Delta$,
$$\rain_< (J_\Delta) + \rain_< (J_{\Delta^\vee} ) = \inn_< I_n (M).$$
\end{obs}

Next we define the \emph{linear strand}, which is one of the central objects of study in this paper. These will be dealt with more intensively in Section \ref{sec:linStrands}.

\begin{definition}\label{def:linStrand}
Let $F_\bullet$ be a minimal graded $R$-free complex with $F_0$ having initial degree $d$. Then the \emph{linear strand} of $F_\bullet$, denoted $F_\bullet^{\lin}$, is the complex obtained by restricting $d_i^F$ to $(F_i)_{d+i}$ for each $i \geq 1$.
\end{definition}

\begin{remark}
Observe that the minimality assumption in Definition \ref{def:linStrand} ensures that the linear strand is well defined. Choosing bases, the linear strand can be obtained by restricting to the columns where only linear entries occur in the matrix representation of each differential.
\end{remark}

The following result, due to Boocher, shows that with respect to any term order $<$, the ideal $\inn_< I_n (M)$ specializes to the ideal of all squarefree monomials of degree $n$ in $m$ variables. 

\begin{theorem}[{\cite[Proof of Theorem 3.1]{boocher2012}}]\label{thm:boocherBnos}
For any term order $<$, the sequence of variable differences
    $$\{ x_{11} - x_{21} , \dots, x_{11} - x_{n1} \} \cup \cdots \cup \{ x_{1m} - x_{2m} , \dotsc , x_{1m} - x_{nm} \}$$
    forms a regular sequence on $R/ \inn_< I_n (M)$. In particular,
    $$\beta_{ij} (R / I_n(M)) = \beta_{ij} (R/ \inn_< I_n (M)) \ \textrm{for all} \ i,j.$$
\end{theorem}

The following definition is the definition of the well-known Eagon-Northcot complex. Recall that the Eagon-Northcott complex is a minimal free resolution of the quotient ring defined by the ideal of maximal minors $I_n (M)$; it will turn out to be the basis from which all resolutions of $R/\inn_< I_n (M)$ can be built.

\begin{definition}[Eagon-Northcott complex]\label{def:EN}
Let $\phi : F \to G$ be a homomorphism of free modules of ranks $n$ and $m$, respectively, with $n \geq m$. Let $c_\phi$ be the image of $\phi$ under the isomorphism $\hom_R (F,G) \xrightarrow{\cong} F^* \otimes G $. The \emph{Eagon-Northcott complex} is the complex
$$0 \to D_{m-n} (G^*) \otimes \bigwedge^m F \to D_{m-n-1} (G^*) \otimes \bigwedge^{m-1} F \to \cdots \to G^* \otimes \bigwedge^{n+1} F \to \bigwedge^n F \to \bigwedge^n G$$
with differentials in homological degree $\geq 2$ induced by multiplication by the element $c_\phi \in F^* \otimes G$, and the map $\bigwedge^g F \to \bigwedge^g G$ is $\bigwedge^g \phi$. 
\end{definition}

\begin{notation}\label{not:ENgrading}
Let $E_\bullet$ denote the Eagon-Northcott complex of Definition \ref{def:EN}. If $F$ has basis $f_1 , \dots , f_m$ and $G$ has basis $g_1 , \dots , g_n$, then define
$$g^{*(\alpha)} \otimes f_I := g_1^{*(\alpha_1)} \cdots g_n^{*(\alpha_n)} \otimes f_{i_1} \w \cdots \w f_{i_{n+\ell}},$$
where $\alpha = (\alpha_1 , \dots , \alpha_n)$ and $I = (i_1 < \cdots < i_{n+\ell})$. Observe that $E_\bullet$ inherits a $\ZZ^n \times \ZZ^m$-grading by setting
$$\mdeg (g^{* (\alpha)} \otimes f_I ) = (1 + \alpha_1 \epsilon_1 + \cdots + \alpha_n \epsilon_n , \epsilon_{i_1} + \cdots \epsilon_{i_{n+\ell}}),$$
where $\epsilon_k$ denotes the appropriately sized vector with $1$ in the $i$th spot and $0$ elsewhere, and $1$ denotes a length $n$ vector of $1$s.
\end{notation}

\begin{cor}\label{cor:Bnos0or1}
Let $F_\bullet$ denote a multigraded resolution of $\inn_< I_n(M)$. Then for every multidegree $\alpha$,
$$\beta_{\alpha} (R / \inn_< I_n(M)) \leq 1.$$
\end{cor}

\begin{proof}
By Theorem \ref{thm:boocherBnos}, a minimal free resolution of $R/\inn_< I_n (M)$ may be obtained by setting some of the entries in the matrix representation of the differentials of Definition \ref{def:EN} equal to $0$. With respect to the $\ZZ^n \times \ZZ^m$-grading of the Eagon-Northcott complex $E_\bullet$, one has 
$$\beta_\alpha (R / \inn_< I_n (M)) \leq 1.$$
Since any $\ZZ^{nm}$-graded minimal free resolution is also $\ZZ^n \times \ZZ^m$-graded, the result follows.
\end{proof}

\begin{definition}
    Let $R = k[x_{ij} \mid 1 \leq i \leq n, \ 1 \leq j \leq m]$ be a polynomial ring over an arbitrary field $k$, endowed with some term order $<$. The \emph{sparse Eagon-Northcott} $E_\bullet^<$ complex is the complex obtained by:
    \begin{enumerate}
        \item homogenizing $E_\bullet$ with a new variable $t$, with respect to the term order $<$, then
        \item setting $t=0$ in the resolution.
    \end{enumerate}
\end{definition}

As observed by Boocher, the minimal free resolution of $R/\inn_< I_n (M)$ may be obtained as the sparse Eagon-Northcott complex $E_\bullet^<$.

\begin{definition}
    Let $F$ be a free $R$-module with some fixed basis $B \subset F$. The \emph{support} of an element $f = \sum_{b \in B} c_b b \in F$ with respect to the basis $B$ is defined to be:
    $$\supp_B (f) := \{ b \in B \mid c_b \neq 0 \},$$
    where $c_b \in R$. When the basis $B$ is understood, the notation $\supp (f)$ will be used instead.
\end{definition}

When dealing with the sparse Eagon-Northcott complex $E_\bullet^<$, the support of any element will be assumed to be with respect to the basis elements of Notation \ref{not:ENgrading}.

\begin{obs}
Let $<$ be any term order. Then the associated sparse Eagon-Northcott complex can be given a $\ZZ^{nm}$-multigrading. 
\end{obs}

\begin{proof}
Assign the multidegrees inductively. In homological degree $1$, $f_I$ has multidegree $\inn_< ([I])$; for basis elements in higher homological degree, the multidegree is assigned as the lcm of the multidegrees appearing in the support of the image under the differential.
\end{proof}

\begin{notation}
Let $<$ be a term order. Every basis element of the sparse Eagon-Northcott complex $E_\bullet^<$ can be assigned a \emph{unique} multidegree
$$g^{*(\alpha)} \otimes f_I \xleftrightarrow{} \prod_{i=1}^n \Big( x_{i b_{|\alpha_{\leq i-1}| + i}} \cdots x_{i b_{|\alpha_{\leq i}|+i}} \Big), $$
where $I  = \{ b_1 , \dots , b_{n+\ell} \}$. For ease of notation, the above multidegree will be written
$$\prod_{i=1}^n \Big( x_{i b_{|\alpha_{\leq i-1}| + i}} \cdots x_{i b_{|\alpha_{\leq i}|+i}} \Big) = x_{1 \bb_1} \cdots x_{n \bb_n},$$
where $\bb_i = ( b_{|\alpha_{\leq i-1}| + i} , \dots , b_{|\alpha_{\leq i}|+i} )$.
\end{notation}

\begin{theorem}\label{thm:sparseEN}
    Let $<$ be any term order. Then the complex $E_\bullet^<$ has differentials of the form
    $$g^{* (\alpha)} \otimes f_I \mapsto \sum_{\substack{i \in \supp (\alpha) \\
    j \in \bb_i \\}} \sgn (j \in I) x_{ij} g^{* (\alpha - \epsilon_i)} \otimes f_{I \backslash j},$$
    where $\mdeg (g^{* (\alpha)} \otimes f_I ) = x_{1 \bb_1} \cdots x_{n \bb_n}$, and
    $$f_I \mapsto \inn_< ([I]).$$
\end{theorem}

\begin{proof}
This follows by induction on the homological degree, where the base case for $\ell=2$ is clear. Let $\ell > 2$; since $(E_\bullet^< , d_\bullet)$ is a complex, by definition $d_{\ell-1} \circ d_\ell = 0$. Assume for sake of contradiction that $\supp (d_\ell (g^{* (\alpha)} \otimes f_I )$ does not contain $g^{* (\alpha - \epsilon_i)} \otimes f_{I \backslash j}$ for some $i \in \supp (\alpha)$, $j \in \bb_i$. Let $r$ be any integer and $s \in \bb_r$ such that $g^{*(\alpha - \epsilon_r)} \otimes f_{I \backslash s } \in \supp (d_\ell (g^{* (\alpha)} \otimes f_I ))$; by the inductive hypothesis, $g^{*(\alpha - \epsilon_r- \epsilon_i)} \otimes f_{I \backslash s,j } \in \supp (g^{*(\alpha - \epsilon_r)} \otimes f_{I \backslash s })$. However, if $g^{* (\alpha - \epsilon_i)} \otimes f_{I \backslash j} \notin \supp (d_\ell (g^{* (\alpha)} \otimes f_I ))$, then the term $x_{ij} x_{rs} g^{*(\alpha - \epsilon_r- \epsilon_i)} \otimes f_{I \backslash s,j }$ does not cancel with any other term in the image $d_{\ell-1} \circ d_\ell ( g^{* (\alpha)} \otimes f_I)$, contradicting the fact that $E_\bullet^<$ is a complex.
\end{proof}

Theorem \ref{thm:sparseEN} immediately yields an explicit description of the multidegrees appearing in each homological degree. This result will be essential later for showing that the associated poset forms a meet semilattice.

\begin{cor}\label{cor:theBasisElts}
The multidegrees appearing in homological degree $\ell >0$ of the complex $E_\bullet^<$ may be constructed by choosing all degree $n+ \ell -1$ monomials $x_{1 \bb_1} \cdots x_{n \bb_n}$ with the property that
$$x_{1 b_{i_1}} x_{2 b_{i_2}} \cdots x_{n b_{i_n}} \in G (\inn_< I_n (M))$$
for all $b_{i_j} \in \bb_j$, $j=1 , \dots , n$. 
\end{cor}

\begin{proof}
This follows by induction on the homological degree combined with the differential of Theorem \ref{thm:sparseEN}.
\end{proof}

To conclude this section, we record the following observations which will be used to prove that bounded intervals in the poset associated to the sparse Eagon-Northcott complexes are in fact CL-shellable.

\begin{obs}\label{obs:semimodOb}
Let $<$ be any term order. Let $1 \leq a_1 , \dots , a_n \leq m$ and set $\mu := x_{1a_1 } x_{2 a_2} \cdots x_{n a_n}$. Assume that $b$ and $c$ are integers such that
$$\frac{\mu x_{ib}}{x_{ia_i}} \quad \textrm{and} \quad \frac{\mu x_{jb}}{x_{j a_j}} \in G ( \inn_< I_n ( M)), \ \textrm{where} \ i \neq j.$$
Then $\frac{\mu x_{i b} x_{j c}}{x_{i a_i} x_{j a_j}} \in G( \inn_< I_n (M))$.
\end{obs}

\begin{proof}
Let $L = \{ a_1 , \dots , \widehat{a_i} , b , \dots , a_j , c , \dots , a_n \}$ (it is not assumed that these are increasing), and consider the standard Eagon-Northcott relation
$$\sum_{\ell \in L } \sgn (\ell) x_{j \ell} [ L \backslash \ell ] = 0.$$
Taking lead terms in the above yields a relation on the initial ideal. Observe that after taking lead terms, the term $x_{jc} \cdot \frac{\mu x_{i b}}{x_{i a_i}}$ appears in the above sum. The only other term that could possibly cancel with this term is the term appearing with coefficient $x_{j a_j}$, which must be $\frac{\mu x_{i b} x_{j c}}{x_{i a_i} x_{j a_j}}$. 
\end{proof}

\begin{obs}\label{obs:AtomOrderobs}
Let $<$ be any term order. Let $1 \leq a_1 , \dots , a_n \leq m$ and $1 \leq b_1 , \dots , b_n \leq m$, and set
$$\mu = x_{1 a_1} \cdots x_{n a_n}, \qquad \nu = x_{1 b_1} \cdots x_{n b_n}.$$
Assume that $\lcm ( \mu , \nu)$ is a valid multidegree as in Corollary \ref{cor:theBasisElts}. Then there exists $1 \leq i \leq n$ such that $\frac{x_{i a_i} \nu}{x_{i b_i}} \in G( \inn_< I_n (M))$ \emph{and} $\frac{x_{i a_i} \nu}{x_{i b_i}} < \nu$. 
\end{obs}

\begin{proof}
Recall that any term order $<$ on $R$ may be obtained as the ordering corresponding to a weight vector $\omega \in \bbz^{nm}$ (see, for instance, \cite[Proposition 1.11]{sturmfels1996grobner}). Observe that if $\lcm ( \mu , \nu)$ is a valid multidegree as in Corollary \ref{cor:theBasisElts}, then for all $1 \leq i \leq n$, one has $\frac{x_{i a_i} \nu}{x_{i b_i}} \in G( \inn_< I_n (M))$. Thus it suffices to show that $i$ may be chosen such that $\frac{x_{i a_i} \nu}{x_{i b_i}} < \nu$.

In terms of the weight vector $\omega$, the assumption $\mu < \nu$ implies $e( \mu) \cdot \omega < e( \nu) \cdot \omega$ (where $e(\cdot)$ denotes the corresponding exponent vector); this means proving the statement is equivalent to proving the following:
\begin{center}
    If $p_1 + \cdots + p_\ell <0$ for some integers $p_i \in \bbz$, then $p_t< 0$ for some $1 \leq t \leq \ell$.
\end{center}
But this is trivially true, whence the statement follows.
\end{proof}

\section{On the Construction of Linear Strands}\label{sec:linStrands}

In this section, we consider the construction of linear strands. We first formulate a more general statement on constructing the linear strand of certain classes of ideals with the machinery of \emph{iterated trimming complexes}. With this more general result, we consider ideals whose linear strand is supported on a regular CW or cell complex. After developing a fair amount of machinery, this section culminates in Theorems \ref{thm:cellularityMainRes} and \ref{thm:disBnosCell}. These results show that ideals whose linear strand is supported on a \emph{linearly connected} cell/CW complex have the property that every subideal is supported on a cell/CW complex that is simple to describe. 

We begin this section with a general result on linear strands originally due to Herzog, Kiani, and Madani; intuitively, this says that a linear complex arises as a linear strand if and only if the homology is concentrated in sufficienly large degrees.

\begin{theorem}[\cite{herzog2015linear}, Theorem 1.1]\label{thm:linstrandequiv}
Let $R$ be a standard graded polynomial ring over a field $k$. Let $G_\bullet$ be a finite linear complex of free $R$-modules with initial degree $n$. Then the following are equivalent:
\begin{enumerate}
    \item The complex $G_\bullet$ is the linear strand of a finitely generated $R$-module with initial degree $n$.
    \item The homology $H_i (G_\bullet)_{i+n+j} = 0$ for all $i>0$ and $j=0, \ 1$.
\end{enumerate}
\end{theorem}

We will temporarily adopt the following general setup for the next two results. In this setup, we are choosing specific generators of a given ideal $I$ that we wish to trim off, then ``rescale" by some new ideal. Theorem \ref{itres} shows how a (generally nonminimal) resolution can be computed for this new ideal based on the free resolutions of all other ideals involved.

\begin{setup}\label{setup4}
Let $R$ be a standard graded polynomial ring over a field $k$. Let $I \subseteq R$ be a homogeneous ideal and $(F_\bullet, d_\bullet)$ denote a homogeneous free resolution of $R/I$. 

Write $F_1 = F_1' \oplus \Big( \bigoplus_{i=1}^t Re_0^i \Big)$, where, for each $i=1, \dotsc , t$, $e^i_0$ generates a free direct summand of $F_1$. Using the isomorphism
$$\hom_R (F_2 , F_1 ) = \hom_R (F_2,F_1') \oplus \Big( \bigoplus_{i=1}^t \hom_R (F_2 , Re^i_0) \Big)$$
write $d_2 = d_2' + d_0^1 + \cdots + d^t_0$, where $d_2' \in \hom_R (F_2,F_1')$ and $d^i_0 \in \hom_R (F_2 , Re^i_0)$. 

For each $i=1, \dotsc , t$, let $\mfa_i$ denote any homogeneous ideal with
$$d^i_0 (F_2) \subseteq \mfa_i e^i_0,$$
and $(G^i_\bullet , m^i_\bullet)$ be a homogeneous free resolution of $R/\mfa_i$. 

Use the notation $K' := \im (d_1|_{F_1'} : F_1' \to R)$, $K^i_0 := \im (d_1|_{Re^i_0} : Re^i_0 \to R)$, and let $J := K' + \mfa_1 \cdot K^1_0+ \cdots + \mfa_t \cdot K_0^t$.
\end{setup}

\begin{theorem}\label{itres}
Adopt notation and hypotheses as in Setup \ref{setup4}. Then there is a morphism of complexes
\begin{equation}\label{itcomx}
\xymatrix{\cdots \ar[r]^{d_{k+1}} &  F_{k} \ar[dd]^{Q_{k-1}}\ar[r]^{d_{k}} & \cdots \ar[r]^{d_3} & F_2 \ar[rrrr]^{d_2'} \ar[dd]^{Q_1} &&&& F_1' \ar[dd]^{d_1} \\
&&&&&&& \\
\cdots \ar[r]^-{\bigoplus m^i_k} & \bigoplus_{i=1}^t G^i_{k-1} \ar[r]^-{\bigoplus m^i_{k-1}} & \cdots \ar[r]^-{\bigoplus m^i_2} & \bigoplus_{i=1}^t G^i_1 \ar[rrrr]^-{-\sum_{i=1}^t m^i_1(-)\cdot d_1(e^i_0)} &&&& R. \\}\end{equation}
Moreover, the mapping cone of \ref{itcomx} is a free resolution of $R/J$.
\end{theorem}

\begin{definition}\label{def:ittrimcx}
The \emph{iterated trimming complex} associated to the data of Setup \ref{setup4} is the complex of Theorem \ref{itres}.
\end{definition}

Restricting the diagram \ref{itcomx} to the linear strands of all complexes involved, one immediately obtains the following:

\begin{cor}\label{cor:howtogetLinStrand}
    Adopt notation and hypotheses as in Setup \ref{setup4}. Assume that the complexes $F_\bullet$, $G_\bullet^j$ ($j=1, \dots , t$) are minimal and that $I$ has initial degree $\ell$. For $i > 1$, let 
    $$C_i := \ker \Big( (Q_{i-1})_{i+\ell} : (F_i)_{i+\ell} \to \bigoplus_{j=1}^t (G_i^j)_{i+\ell} \Big).$$ 
    Then the complex
    $$\cdots \xrightarrow{d_{k+1}} C_k \xrightarrow{d_k} \cdots \xrightarrow[]{d_3} C_2 \xrightarrow[]{d_2'} F_1' \to 0$$
    is the linear strand of $J$.
\end{cor}

Observe that if one chooses $\mfa_i := (K' : K_0^i)$ in Setup \ref{setup4}, the the mapping cone of \ref{itcomx} will be a (necessarily nonminimal) resolution of $R/K'$. This fact is what will allow us to construct linear strands for subideals of a given ideal $K$.

For the remainder of this section, we will transition into dealing with cell/CW complexes supporting linear complexes. We will have to build up a decent amount of machinery before constructing a morphism of complexes whose kernel can recover the linear strand of certain subideals. In some of the literature, the terms cell complex or CW complex are used interchangeably; we will only use cellular to refer to \emph{polyhedral cell complexes}, and CW will mean \emph{regular CW-complex}. We begin with the definition of a polyhedral cell complex; for more precise details, see \cite[Chapter 6.2]{BH98-CMrings} or \cite[Chapter 4]{miller2004combinatorial}. 

\begin{definition}\label{def: cellComplex} A \emph{polyhedral cell complex} $\cP$ is a finite collection of convex polytopes (called \emph{cells} or \emph{faces} of $\cP$) in some Euclidean space, satisfying the following two properties:
\begin{itemize}
    \item if $H$ is a polytope in $\cP$, then every face of $H$ also lies in $\cP$, and
    \item if $H_i, H_j$ are both in $\cP$, then $H_i\cap H_j$ is a face of both $H_i$ and $H_j$.
\end{itemize}
Denote by $V(\cP)$ the set of vertices (or $0$-dimensional cells) of $\cP$.
If $\cX\subseteq V(\cP)$, the \emph{induced subcomplex} of $\cP$ on $\cX$ is the subcomplex $\{F\in \cP \mid V(F)\subseteq \cX \}$.
The $f$-vector of a $d$-dimensional polyhedral cell complex $\cP$ is the vector $(f_0,f_1,\dots, f_{d} )$, where $f_i$ is the number of $i$-dimensional cells of $\cP$.
\end{definition}

For the definition of a regular CW-complex, the reader is referred to \cite{hatcher2005algebraic}. Alternatively, one can define an object to be a regular CW complex if its associated face poset os a CW poset, as in Definition \ref{def:CWposet}. 

Let $\cP$ denote a cell/CW complex supporting a linear complex. Let $v$ be a vertex of $\cP$ and let $N_v$ denote the neighbors of $v$; that is, all vertices connected to $v$ in the graph of linear syzygies ($v$ is not considered a neighbor of itself). The incidence function associated to any given cell/CW complex will be denoted by $\epsilon$.

Lemmas \ref{lem:injCorr} and \ref{lem:existenceLem} will be essential for establishing the existence so-called support chains centered at a vertex (see Definition \ref{def:suppChain}). For the remainder of this section, the results will be stated only for regular CW complexes with the understanding that they hold for polyhedral cell complexes as well.

\begin{lemma}\label{lem:injCorr}
Let $\cP$ be a CW complex supporting a linear complex. Then the correspondence
$$P \mapsto V(P)$$
is injective.
\end{lemma}

\begin{proof}
Assume $P$ and $Q$ are faces of $\cP$ with the property that $V(P) = V(Q)$. Observe that since the multidegree of the basis elements $e_P$ and $e_Q$ are completely determined by their vertex sets, $P$ and $Q$ must have the same dimension. Since $\cP$ is a CW complex, $P \cap Q$ is a face of strictly smaller dimension such that $e_{P \cap Q}$ has the same multidegree. Since $\cP$ is linear, no such basis element can exist. 
\end{proof}

\begin{lemma}\label{lem:existenceLem}
Let $\cP$ be a CW complex supporting a linear complex. Then for any vertex $p \in V(P)$, there exists a unique codimension $1$ face $Q \subset P$ with $p \notin V(Q)$.
\end{lemma}

\begin{proof}
Observe first that for any codimension $1$ face $Q \subset P$, $|V(Q)| = |V(P)|-1$. This follows from the fact that if there were at least $2$ vertices not contained in $V(Q)$, then the induced CW complex differential would have degree at least $2$. Thus, if any $Q$ as in the statement were to exist, it must be unique by Lemma \ref{lem:existenceLem}.

It remains to show that such a $Q$ exists. One may assume without loss of generality that there exists some codimension $2$ face $R$ with $p \notin V(R)$, since otherwise every face of $P$ must contain $p$. Since $\cP$ is a CW complex, $R$ may be written as the intersection of two codimension $1$ faces of $P$. This implies that one of these faces \emph{cannot} contain $p$, whence the result.
\end{proof}

\begin{notation}
Let $R$ be a standard graded polynomial ring over a field $k$ endowed with some term order $<$. Let $v \in \cP$ be a vertex of some CW complex supporting a linear complex. For every $v' \in N_v$, define
$$x_{v'} := \frac{m_{v,v'}}{m_v}.$$
There is an induced order on elements of $N_v$ defined by
$$v_1 < v_2 \iff x_{v_1} < x_{v_2}, \quad v_1 , v_2 \in N_v.$$
\end{notation}

\begin{lemma}\label{lem:supportChains}
Let $\cP$ be a CW complex supporting a linear complex and let $v$ be any vertex of $\cP$. Let $P$ be an $n$-dimensional face of $\cP$ with $V(P) \cap N_v = \{ v_1 < \cdots < v_n \}$. Then there exists a \emph{unique} chain
$$P = P_n \supset P_{n-1} \supset \cdots \supset P_1 \supset P_0 = v$$
such that:
\begin{enumerate}
    \item $P_i$ is a codimension $1$ face of $P_{i+1}$ for each $i=0, \dots , n-1$, and
    \item $P_i \cap N_v = \{ v_{n-i} < \cdots < v_{n} \}$.
\end{enumerate}
\end{lemma}

\begin{proof}
This is just an iteration of Lemmas \ref{lem:injCorr} and \ref{lem:existenceLem}. 
\end{proof}

\begin{definition}\label{def:suppChain}
Let $P$ be an $n$-dimensional face of $\cP$ with $V(P) \cap N_v = \{ v_1 < \cdots < v_n \}$. Then the \emph{support chain centered at $v$} associated to $P$ is the unique chain of subfaces of Lemma \ref{lem:supportChains}.

Given a face $P$ as above, define
$$c(P) := \epsilon (P_n , P_{n-1}) \cdot \epsilon (P_{n-1} , P_{n-2}) \cdots \epsilon (P_0 , P_1),$$
where 
$$P = P_n \supset P_{n-1} \supset \cdots \supset P_1 \supset P_0 = v$$
is the associated support chain centered at $v$.
\end{definition}

\begin{lemma}\label{lem:theSignsWork}
Let $P$ be an $n$-dimensional face of $\cP$ with $V(P) \cap N_v = \{ v_1 < \cdots < v_n \}$. Let $Q$ be the unique codimension $1$ face of $P$ such that $v_i \notin V(Q)$ for some $1 \leq i \leq n$. Then,
$$\epsilon (Q,P) c(Q) = (-1)^{i+1} c(P).$$
\end{lemma}

\begin{proof}
Let 
$$ P = P_n \supset P_{n-1} \supset \cdots \supset P_1 \supset P_0 = v \quad \textrm{and}$$
$$ Q = Q_{n-1} \supset Q_{n-2} \supset \cdots \supset Q_1 \supset Q_0 = v$$
be the support chains centered at $v$ associated to $P$ and $Q$, respectively. Observe that $P_{n-i+1} = Q_{n-i+1}$; moreover, since $P_{n-i+1}$ is a codimension $2$ face of $Q_{n-i+3}$, $Q_{n-i+2}$ and $P_{n-i+2}$ must be the unique codimension $1$ faces of $Q_{n-i+3}$ containing $P_{n-i+1}$. By definition of the incidence function,
$$\epsilon ( Q_{n-i+2} , Q_{n-i+1} ) \epsilon (P_{n-i+1} , Q_{n-i+2}) = -\epsilon ( P_{n-i+2} , Q_{n-i+1} ) \epsilon (P_{n-i+1} , P_{n-i+2}).$$
Repeating this argument $i$ more times, the result follows.
\end{proof}

With the existence of support chains established, there are a few more steps required for defining the relevant morphism of complexes. The most important of these steps is Lemma \ref{lem:maxlSupp}, which will imply that the morphism of Lemma \ref{lem:theQmap} is well-defined.

\begin{obs}
Let $\cP$ be a CW complex supporting the linear strand of some ideal $(I,g)$, where $g \in R$. Then the linear strand of $(I:g)$ is obtained as the linear strand of the ideal $(x_{v'} \mid v' \in N_v)$.
\end{obs}

\begin{setup}\label{set:CellSetup}
Let $R$ be a standard graded polynomial ring over a field endowed with some term order $<$. Let $\cP$ be a CW complex supporting a linear complex $F_\bullet$ and suppose that $v \in \cP$ is some vertex. Let $K_\bullet$ denote the Koszul complex resolving $R/(x_{v'} \mid v' \in N_v )$.
\end{setup}

\begin{lemma}\label{lem:maxlSupp}
Let $\cP$ be a CW complex supporting a linear complex. Let $P$ be an $n$-dimensional face of $\cP$ supported on $v$. Then, $|V(P) \cap N_v| = n$.
\end{lemma}

In the following proof, we will say that $P$ is \emph{maximally supported on $v$} if $|V(P) \cap N_v | =n$, as in the statement of the Lemma.

\begin{proof}
The proof follows by induction on the dimension of any given face of $\cP$. If $P$ is a $1$-dimensional face containing $v$, then by definition $P$ contains a neighbor of $v$. Let $P$ be an $i$-dimensional face and assume by induction that all faces of dimension $\leq i-1$ containing $v$ are maximally supported on $v$. By Lemma \ref{lem:existenceLem}, there exists a codimension $1$ face containing $v$ that is also not maximally supported on $v$, which is a contradiction.
\end{proof}

Lemma \ref{lem:theQmap} defines the morphism of complexes whose kernel can be used to compute the linear strand of certain subideals, as guaranteed by Corollary \ref{cor:howtogetLinStrand}.

\begin{lemma}\label{lem:theQmap}
Adopt notation and hypotheses as in Setup \ref{set:CellSetup}. Define
\begingroup\allowdisplaybreaks
\begin{align*}
    Q_{n-1}: F_n & \to K_{n-1} \\
    e_P & \mapsto \begin{cases}
    c(P) e_{V(P) \cap N_v} & \textrm{if} \ v \in V(P) \\
    0 & \textrm{otherwise}. \\
    \end{cases}
\end{align*}
\endgroup
Then for all $i \geq 2$, the following diagram commutes:
\begin{equation}\label{diag:theDiagram}
    \xymatrix{F_{i+1} \ar[d]_{Q_i} \ar[r] & F_i \ar[d]^{Q_{i-1}} \\
K_i \ar[r] & K_{i-1} \\}
\end{equation}
\end{lemma}

\begin{proof}
Observe first that each $Q_i$ is well defined by Lemma \ref{lem:maxlSupp}. Going clockwise around the diagram \ref{diag:theDiagram}, one obtains:
\begingroup\allowdisplaybreaks
\begin{align*}
    e_P &\mapsto \sum_{Q \subset P} \epsilon(Q , P) m_P / m_Q e_Q \\
    &\mapsto \sum_{|V(Q) \cap N_v| = n-1} \epsilon (Q , P) c(Q) m_P / m_Q e_{V(Q) \cap N_v}. 
\end{align*}
\endgroup
Going counterclockwise around \ref{diag:theDiagram},
\begingroup\allowdisplaybreaks
\begin{align*}
    e_P & \mapsto c(P) e_{V(P) \cap N_v} \\
    &\mapsto \sum_{i=1}^n (-1)^{i+1} c(P) x_{v_{j_i}} e_{ V(P) \cap N_v \backslash v_{j_i}} . 
\end{align*}
\endgroup
The result then follows from Lemma \ref{lem:theSignsWork}.
\end{proof}

In the case of polyhedral cell complexes, it turns out that the kernels of the maps $Q_i$ as above may \emph{not} yield a complex supported on a polyhedral cell complex in general. The following Definition will end up being a property sufficient for guaranteeing that the resulting kernel can be easily described as an induced subcomplex. 

\begin{definition}
    A CW complex $\cP$ is \emph{linearly connected with respect to $v$} if for all $v_i , v_j \in N_v$, $(v_i , v_j)$ is an edge of $\cP$ whenever $m_{v_i}$ and $m_{v_j}$ have a linear syzygy.
\end{definition}

\begin{lemma}\label{lem:linConn}
Let $\cP$ be a CW complex supporting the linear strand of some module. If $\cP$ is linearly connected with respect to $v$, then $x_{v_i} \neq x_{v_j}$ for all distinct $v_i, \ v_j \in N_v$.
\end{lemma}

\begin{proof}
Let $v_1, v_2 \in N_v$ be such that $x_{v_1} = x_{v_2}$ and suppose for sake of contradiction that $\{ v_1 , v_2 \}$ is a face of $\cP$. One may assume that the differentials in the presenting matrix have the following form:
\begingroup\allowdisplaybreaks
\begin{align*}
    e_{v,v_1} &\mapsto \frac{m_{v,v_1}}{m_v} e_v -  \frac{m_{v,v_1}}{m_{v_1}} e_{v_1}, \\
    e_{v,v_2} &\mapsto \frac{m_{v,v_2}}{m_v} e_v -  \frac{m_{v,v_2}}{m_{v_2}} e_{v_2}, \\
    e_{v_1,v_2} &\mapsto \frac{m_{v_1,v_2}}{m_{v_1}} e_{v_1} -  \frac{m_{v_1,v_2}}{m_{v_2}} e_{v_2}. \\
\end{align*}
\endgroup
Since $x_{v_1} = x_{v_2}$, one has $m_{v, v_1} = m_{v,v_2} = m_{v_1,v_2}$. It is then clear that $e_{v,v_1} - e_{v,v_2} + e_{v_1 , v_2}$ is a cycle that cannot possibly be a boundary. By Theorem \ref{thm:linstrandequiv}, no such element can exist.
\end{proof}

\begin{theorem}\label{thm:cellularityMainRes}
Let $\cP$ be a CW complex supporting the linear strand of some monomial ideal $I$. Assume that $\cP$ is linearly connected with respect to $v$. Then the CW complex induced by removing all faces supported on $v$ is the linear strand of the monomial ideal with minimal generating set $G(I) \backslash m_v$. 
\end{theorem}

\begin{proof}
Let $J$ denote the ideal with minimal generating set $G(I) \backslash m_v$. By Corollary \ref{cor:howtogetLinStrand}, the linear strand of $J$ is obtained by restricting to the kernel of each $Q_i$ map as in Lemma \ref{lem:theQmap}. Moreover, for every $V \subset N_v$, there is \emph{at most} one face with support containing $V$; if there were more than one such face, then the intersection of these faces would also be supported on $V$, a contradiction to Lemma \ref{lem:maxlSupp}. Thus, by Lemma \ref{lem:linConn}, the only possible basis elements contained in the kernel are those basis elements corresponding to faces that do not contain $v$. 
\end{proof}

\begin{theorem}\label{thm:disBnosCell}
Let $\cP$ be a CW complex supporting the linear strand of some monomial ideal $I$, and assume that every $1$-dimensional face of $\cP$ has a unique multidegree. Then the linear strand of any monomial ideal $J$ with $G(J) \subseteq G(I)$ is supported on the CW complex induced by restricting to the vertices corresponding to generators of $J$.

In particular, if $J$ has linear minimal free resolution, then $J$ has a CW minimal free resolution.
\end{theorem}

\begin{proof}
The assumption that every $1$-dimensional face has a unique multidegree implies that for any vertex $v \in \cP$, every pair of distinct vertices $v_1 , v_2 \in N_v$ has the property that $x_{v_1} \neq x_{v_2}$. The proof then proceeds identically as in the proof of Theorem \ref{thm:cellularityMainRes}.
\end{proof}

\section{CW Posets and Cellularity of Sparse Eagon-Northcott Complexes}\label{sec:latticesAndCell}

In this section, we consider the question of when a particular complex may be supported on a CW complex. We first recall the construction of so-called \emph{frames} of monomial ideals, as described by Peeva and Velasco in \cite{peeva2011frames}. We then show that the poset $\po (E_\bullet^<)$ associated to the sparse Eagon-Northcott complex is a CW poset, whence Corollary \ref{cor:initisCellular} immediately establishes that $E_\bullet^<$ is CW for \emph{any} term order $<$. Combining this with the results of Section \ref{sec:linStrands}, we obtain the previously mentioned fact that every rainbow DFI has linear strand supported on a CW complex (see Corollary \ref{cor:rainbowisCellular}). 

\begin{definition}
An $r$-frame $U_\bullet$ is a finite complex of finite $k$-vector spaces with differential $\partial$ and a fixed basis satisfying
\begin{enumerate}
    \item $U_0 = k$,
    \item $U_1 = k^r$,
    \item $\partial (w_j) = 1$ for each basis vector $w_j \in U_1$. 
\end{enumerate}
\end{definition}

Recall that for any choice of $r$ monomials $m_1 , \dots , m_r$, an $r$-frame $U_\bullet$ may be homogenized to a (not necessarily acyclic) complex $F_\bullet$ with $H_0 (F_\bullet) = R / (m_1 , \dots , m_r)$. This is accomplished inductively by assigning a multidegree to a basis element $e_P$ by taking the lcm of the multidegrees of the basis elements appearing in the support of $\partial (e_P)$. Once the multigrading has been assigned, the differential may be computed as follows:
$$d (e_P) = \sum_{e_Q \in \supp (\partial (e_P))}  a_{P,Q} \frac{\mdeg (e_P)}{\mdeg(e_Q)} e_Q,$$
where the coefficients $a_{P,Q}$ are determined by the associated $r$-frame. Complete details are found in Construction $55.2$ of \cite{peeva2010graded}.

\begin{theorem}[{\cite[Theorem 4.14]{peeva2011frames}}]\label{thm:frameThm}
Let $R$ be a polynomial ring over a field and $I \subseteq R$ a monomial ideal. Then the $I$-homogenization of any frame of the multigraded minimal free resolution $F_\bullet$ of $R/I$ is $F_\bullet$.
\end{theorem}

\begin{cor}\label{cor:linStrandHom}
Let $R$ be a polynomial ring over a field and $I \subseteq R$ a monomial ideal. Then the $I$-homogenization of any frame of the multigraded linear strand $F_\bullet^{\lin}$ of $R/I$ is $F_\bullet^{\lin}$. 
\end{cor}

\begin{proof}
Viewing $F_\bullet^{\lin}$ as a subcomplex of the minimal free resolution, this follows immediately from Theorem \ref{thm:frameThm}.
\end{proof}

The following Definition is a method of associating any resolution (with a fixed basis in each homological degree) to a poset whose elements correspond to the basis elements. There is a natural way to try to associate any such poset to a face lattice of some geometric object, but in general this will not yield a cell or CW complex that is actually realizable in some Euclidean space. 

\begin{definition}
    Let $F_\bullet$ be a complex with a fixed basis $B_i$ in each homological degree $i$. The associated poset $\po (F_\bullet)$ is defined as follows: given basis elements $e \in F_i$, $f \in F_{i-1}$,
    $$f \leq e \iff f \in \supp_{B_{i-1}} (d^F (f)).$$
    Extending $\leq$ transitively yields a partial order on basis elements appearing in all homological degrees.
\end{definition}

\begin{example}
Let $\psi : F \to R$ be an $R$-module homomorphism, where $F$ is a free $R$-module. The induced Koszul complex is the complex $K_\bullet$ with $K_i = \bigwedge^i F$ and differential
\begingroup\allowdisplaybreaks
\begin{align*}
    \bigwedge^i F &\xrightarrow{\textrm{comultiplication}} F \otimes \bigwedge^{i-1} F \\
    & \xrightarrow{\psi \otimes 1} R \otimes \bigwedge^{i-1} F = \bigwedge^{i-1} F. \\
\end{align*}
\endgroup
Let $F$ have basis $f_1 , \dots , f_n$. If one chooses the standard bases for each exterior power $K_i$, the poset $\po (K_\bullet)$ is precisely the power set $2^{[n]}$ partially ordered by inclusion (ie, the Boolean poset).
\end{example}

The following Definition is due to Bj\"orner \cite{bjorner1984posets}.

\begin{definition}\label{def:CWposet}
    A poset $S$ is a \emph{CW-poset} if
    \begin{enumerate}[(a)]
        \item $S$ has a least element $0$,
        \item $S$ has more than $1$ element, and
        \item for all $x \in S \backslash \{ 0 \}$, the open intercal $(0 , x)$ is homeomorphic to a sphere.
    \end{enumerate}
\end{definition}

As previously mentioned, one can take the definition of a regular CW-complex to be such that its associated face poset is CW, by the following Theorem:

\begin{theorem}[{\cite[Proposition 3.1]{bjorner1984posets}}]
A poset $S$ is a CW poset if and only if it is isomorphic to the face poset of a regular CW-complex.
\end{theorem}

\begin{prop}
   Let $S$ be a nontrivial poset with respect to some binary relation $\leq$. If:
   \begin{enumerate}[(a)]
       \item $S$ has a least element, denoted $0$.
       \item Every interval $[x,y]$ of length two has cardinality $4$; that is, $S$ is \emph{thin}.
       \item Every interval $[0,x]$, $x \in S$, is finite and shellable.
   \end{enumerate}
   Then, $S$ is a CW-poset.
\end{prop}

As previously mentioned, choosing a face lattice defined by $\po (F_\bullet)$ does not always yield a CW complex. However, the following lemmas will be used to prove that (with respect to the basis represented by valid multidegrees), the poset associated to the sparse Eagon-Northcott complexes are the face poset of a regular CW-complex. For the definition of CL-shellable in the following result, see \cite[Definition 2.2]{bjorner1983lexicographically}.

\begin{lemma}\label{lem:CLshellable}
Adopt notation and hypotheses as in Theorem \ref{thm:sparseEN}. Then every interval of the form $[0,x]$ in $\po (E_\bullet^<)$ is finite and CL-shellable.
\end{lemma}

\begin{proof}
Recall that a finite poset $S$ is semimodular if, for any two elements $u$ and $v$ covering some $x \in S$, there exists a $z \in S$ covering $u$ and $v$. The poset $S$ is totally semimodular if every interval is semimodular. We will first prove that for any elements $\mu, \nu  \in \po (E_\bullet^<)$ with $\nu \geq \mu$ and $\rank \mu \geq 1$, the interval $[\mu , \nu]$ is totally semimodular.

Let $\mu$ be any multidegree appearing in homological degree $ \geq 1$. The statement that $\mu$ is covered by two elements is equivalent to saying that $x_{i b} \mu$ and $x_{j c} \mu$ are valid multidegrees in $E_\bullet^<$ for some integers $i,j,b,c$. If $i=j$, then Corollary \ref{cor:theBasisElts} immediately yields that $x_{ib} x_{jc} \mu$ is a valid multidegree. If $i \neq j$, then one simply combines Observation \ref{obs:semimodOb} with Corollary \ref{cor:theBasisElts} to conclude that $x_{i b} \mu$ and $x_{j c} \mu$ are covered by some element of $[ \mu , \nu]$.

Next, the proof will be complete by showing that every interval of the form $[0,\nu]$ admits a recursive atom ordering; recall that CL-shellability is equivalent to the existence of a recursive atom ordering by \cite[Theorem 3.2]{bjorner1983lexicographically}. More precisely:

\textbf{Claim:} The ordering induced by the term order $<$ (on the multidegrees of) each atom of the interval $[0 , \nu]$ is a recursive atom ordering.

The claim will follow upon proving the following conditions (these conditions are a translation of the definition of a recursive atom ordering as given in \cite[Definition 3.1]{bjorner1983lexicographically}):
\begin{enumerate}
    \item If $\mu_1 < \mu_2$ are atoms of $\po (E_\bullet^<)$ such that $\lcm (\mu_1 , \mu_2)$ is a valid multidegree as in Corollary \ref{cor:theBasisElts}, then there exists a $\mu_3 < \mu_2$ such that $\mu_2$ and $\mu_3$ are both covered by some element.
    \item For every atom $\mu \in [0 , \nu]$, the interval $[ \mu , \nu]$ admits a recursive atom ordering for which any atom of $[\mu , \nu]$ that is also an atom of $[\mu' , \nu]$ for some $\mu' < \mu$ comes first in the order.
\end{enumerate}

\textbf{Proof of (1):} This is just a retranslation of the statement of Observation \ref{obs:AtomOrderobs}, so there is nothing to prove here.

\textbf{Proof of (2):} As proved above, every interval of the form $[\mu , \nu]$ is totally semimodular. By \cite[Theorem 5.1]{bjorner1983lexicographically}, \emph{every} atom ordering of a totally semimodular poset is a recursive atom ordering. Every atom of $[\mu , \nu]$ is of the form $x_{ia} \mu$ for some variable $x_{ia}$, implying that $x_{ia} \mu$ is \emph{always} an atom of $[\mu' , \nu]$ for $\mu' = \frac{x_{ia} \mu}{x_{ib}}$, where $x_{ib}$ divides $\mu$. Order all atoms of the form $x_{ia} \mu$ with $x_{ia} < x_{ib}$ with the term order $<$, then simply choose all other atoms to come after these terms. By construction this ordering satisfies $(2)$, and the result follows.
\end{proof}

\begin{lemma}\label{lem:sparseENmeetsemi}
Adopt notation and hypotheses as in Theorem \ref{thm:sparseEN}. Then $\po (E_\bullet^<)$ is a CW poset.
\end{lemma}

\begin{proof}
It is clear that $E_\bullet^<$ has a least element, represented by the multidegree $1$. Moreover, $\po (E_\bullet^<)$ is thin because $E_\bullet^<$ is a linear resolution with unique multigraded Betti numbers. Thus the only nontrivial property to check is that $[0, x]$ is shellable for all $x \in \po (E_\bullet^<)$. 

By \cite[Proposition 2.3]{bjorner1983lexicographically}, any CL-shellable poset is also shellable, whence the result follows from Lemma \ref{lem:CLshellable}.
\end{proof}

Combining the previous two results immediately yields the following:

\begin{cor}\label{cor:initisCellular}
Let $R = k[x_{ij} \mid 1 \leq i \leq n, \ 1 \leq j \leq m]$ be a polynomial ring over an arbitrary field $k$ and $M$ be an $n \times m$ matrix of variables in $R$ where $n\leq m$. Let $<$ denote any term order. Then $R/\inn_< I_n (M)$ has minimal free resolution supported on a regular CW complex.
\end{cor}

\begin{proof}
Observe that the frame associated to the sparse Eagon-Northcott complex has coefficients $0$ or $\pm 1$. By Lemma \ref{lem:sparseENmeetsemi}, the frame may be supported on a regular CW complex, and the homogenization of this frame is precisely $E_\bullet^<$. 
\end{proof}

\begin{remark}
For the remainder of the paper, no distinction will be made between the poset $\po (E_\bullet^<)$ and the CW complex that it induces.
\end{remark}

 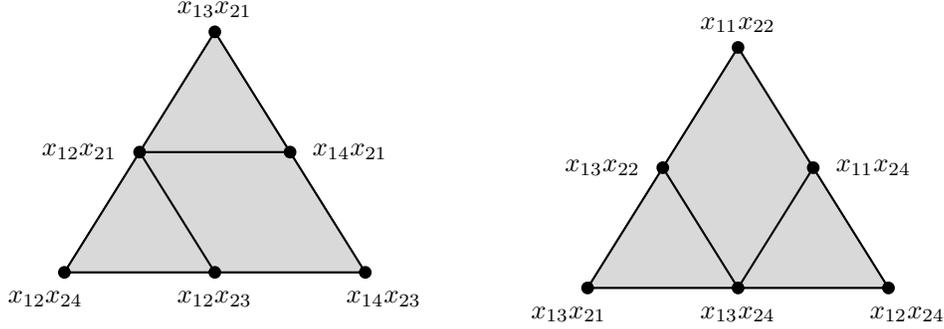
\begin{figure}
    \begin{minipage}{0.45\textwidth}
        \begin{tikzpicture}[]
\draw[fill=gray!30]   (-2,0) -- (2,0) -- (0,3.2) -- (-2,0);
\filldraw[thick]
(-2,0) circle (2pt) --
(0,0) circle (2pt)  -- 
(2,0) circle (2pt) --
(1,1.6) circle (2pt) --
(0,3.2) circle (2pt) --
(-1,1.6) circle (2pt) --
(-2,0);
\draw[thick] (0,0)--(-1,1.6)--(1,1.6);
\node at (-2.25,-0.35) {$x_{12} x_{24}$};
\node at (0,-0.35) {$x_{12} x_{23}$};
\node at (2.25,-0.35) {$x_{14} x_{23}$};
\node at (-1.8,1.6) {$x_{12} x_{21}$};
\node at (1.8, 1.6) {$x_{14} x_{21}$};
\node at (0, 3.5) {$x_{13} x_{21}$};
\end{tikzpicture}
\label{fig: boxCplx} 
    \end{minipage}\hfill
    \begin{minipage}{0.45\textwidth}
        \begin{tikzpicture}[]
\draw[fill=gray!30]   (-2,0) -- (2,0) -- (0,3.2) -- (-2,0);
\filldraw[thick]
(-2,0) circle (2pt) --
(0,0) circle (2pt)  -- 
(2,0) circle (2pt) --
(1,1.6) circle (2pt) --
(0,3.2) circle (2pt) --
(-1,1.6) circle (2pt) --
(-2,0);
\draw[thick] (1,1.6)--(0,0)--(-1,1.6);
\node at (-2.25,-0.35) {$x_{13} x_{21}$};
\node at (0,-0.35) {$x_{13} x_{24}$};
\node at (2.25,-0.35) {$x_{12} x_{24}$};
\node at (-1.8,1.6) {$x_{13} x_{22}$};
\node at (1.8, 1.6) {$x_{11} x_{24}$};
\node at (0, 3.5) {$x_{11} x_{22}$};
\end{tikzpicture}
    \end{minipage}
    \caption{Two examples of cellular resolutions arising from initial ideals of maximal minors of a $2 \times 4$ matrix.}
    \end{figure}

Using the results of Section \ref{sec:linStrands}, Corollary \ref{cor:initisCellular} immediately implies the much stronger result that \emph{every} rainbow DFI has linear strand supported on a CW complex.



\begin{cor}\label{cor:rainbowisCellular}
Adopt notation and hypotheses as in Corollary \ref{cor:initisCellular}. Let $\Delta$ be an $(n-1)$-dimensional pure simplicial complex. Then the linear strand of $\rain_< (J_\Delta)$ is supported on a CW complex. 

In particular, if $\rain_< (J_\Delta)$ has linear resolution, then $\rain_< (J_\Delta)$ has a CW resolution.
\end{cor}

\begin{proof}
This is immediate by Corollary \ref{cor:initisCellular} combined with Theorem \ref{thm:disBnosCell} and Corollary \ref{cor:Bnos0or1}.
\end{proof}

\section{Preservation of Regularity for Subideals and Linearity}\label{sec:regandLinearity}

Corollary \ref{cor:rainbowisCellular} establishes that any rainbow DFI with linear resolution must have a CW resolution. This immediately begs the question of which rainbow DFIs have linear resolution. In this section, we consider conditions that guarantee linearity. We first study rainbow DFIs for which the sequence of variable differences 
$$\{ x_{11} - x_{21} , \dots, x_{11} - x_{n1} \} \cup \cdots \cup \{ x_{1m} - x_{2m} , \dotsc , x_{1m} - x_{nm} \}$$
forms a regular sequence on the associated quotient ring. This can be used to show that certain ideals for which the squarefree complement (see Definition \ref{def:monomIdealDefs}) has no linear relations must have CW resolution. 

Next, using a result of Nagel and Reiner \cite{NR09}, we see that rainbow DFIs with the property that the Alexander dual $\Delta^\vee$ admits an ordering for which it is a \emph{free sequence} (see Definition \ref{def:freeSeq}) must have linear resolution. Moreover, if the elements of the Alexander dual have nonmaximal overlap, this condition is an equivalence; the Betti table for these ideals can also be computed explicitly.

We begin with some definitions and notation related to monomial ideals and their complementary ideals.

\begin{definition}\label{def:monomIdealDefs}
Let $R = k[x_1 , \dots , x_n]$ be a standard graded polynomial ring over a field $k$. Let $K$ denote an equigenerated monomial ideal with generators in degree $d$. Define 
$$G(K) := \textrm{unique minimal generating set of} \ K \ \textrm{consisting of monic monomials.}$$
Given a monomial ideal $K$, define the (squarefree) complementary ideal $\overline{K}$ to be the ideal with minimal generating set:
$$G(\overline{K}) = \begin{cases}
    \{ \textrm{degree} \ d \ \textrm{squarefree monomials} \} \backslash G(K) & \textrm{if} \ K \ \textrm{squarefree}, \\
    \{ \textrm{degree} \ d \ \textrm{monomials} \} \backslash G(K) & \textrm{otherwise.} \\
    \end{cases}$$
\end{definition}

\begin{lemma}[{\cite{vdb2020maxlideal}}]\label{lem:linearResLem}
Let $K'$ be a squarefree equigenerated momomial ideal of degree $n$ with $\overline{K'} = (x^{I_1}, \dots , x^{I_r})$ such that $|I_i \cap I_j| < n-1$. Then $R/K'$ has a linear minimal free resolution.
\end{lemma}

It is worth noting that \cite{vdb2020maxlideal} not only shows that ideals as in the above statement have linear resolution, but also produces the minimal free resolution explicitly. As previously mentioned, we wish to study when the sequence of variable differences
$$\{ x_{11} - x_{21} , \dots, x_{11} - x_{n1} \} \cup \cdots \cup \{ x_{1m} - x_{2m} , \dotsc , x_{1m} - x_{nm} \}$$
forms a regular sequence on the quotient defined by a rain DFI. This problem can be easily abstracted to the following question, which we call the \emph{preservation of regularity}:
\begin{center}
    If $J \subseteq I$ and $\mfa$ is regular on $R/I$, then when is $\mfa$ regular on $R/J$?
\end{center}
The initial answer to this question, although simple, turns out to be surprisingly useful.

\begin{lemma}\label{lem:PresReg}
Let $R$ be a commutative ring and let $\mfa$ be regular on $R/ (J,g)$. Then $\mfa$ is regular on $R/J$ if and only if $\mfa$ is regular on $(J:g)$.
\end{lemma}

\begin{proof}
Without loss of generality assume that $\mfa = a$ has length $1$. The proof follows immediately from the string of inclusions:
$$\ass (J : g) \subseteq \ass (J) \subseteq \ass(J:g ) \cup \ass (J,g).$$


\end{proof}

If we assume furthermore that the ideal $I$ is a monomial ideal with a convenient splitting, then Lemma \ref{lem:PresReg} can be upgraded considerably.

\begin{cor}\label{cor:equiGenCor}
Let $R$ be a polynomial ring over a field. Let $I$ be a linearly presented, equigenerated monomial ideal of degree $d$. Assume that $I = J +K$, where $J$ and $K$ are monomial ideals with $\beta_{1,d+1} (K) = 0$. Let $\mfa$ be regular on $R/I$. 

Then $\mfa$ is regular on $R/J$ if and only if $\mfa$ is regular on $R / (J :g)$ for all $g \in G(K)$.
\end{cor}

\begin{proof}
Write $G(K) = \{ g_1 , \dots , g_t \}$. Since $I$ is linearly presented and $K$ has no linear relations on any of its generators (this is the assumption $\beta_{1,d+1} (K) = 0$), one finds
$$(J + (g_1 , \dots , g_{i} ) : g_{i+1} ) = (J : g_{i+1} ) \ \textrm{for all} \ i=0, \dots , t-1.$$
The result then follows by iterating Lemma \ref{lem:PresReg}.
\end{proof}

The following Proposition is a standard exercise which we include here for convenience.

\begin{prop}\label{prop:stdExercise}
    Let $M$ be a finitely generated $R$-module. Assume $\mfa$ is an $M$-regular sequence and let $F_\bullet$ be any free resolution of $M$ over $R$. Then $F_\bullet \otimes R/\mfa$ is a free resolution of $M/\mfa M$ over $R/\mfa$.
\end{prop}

\begin{proof}
By induction it is of no loss of generality to assume $\mfa = a$ has length $1$. There is a short exact sequence of complexes:
$$0 \to F_\bullet \xrightarrow{a} F_\bullet \to F_\bullet / aF_\bullet \to 0.$$
The long exact sequence of homology implies $H_i (F_\bullet / a F_\bullet) = 0$ for $i \geq 2$, and $H_1 (F_\bullet / a F_\bullet) = 0$ by the assumption that $a$ is regular on $M$. 
\end{proof}

\begin{setup}\label{set:rainDFIsetup}
Let $R = k[x_{ij} \mid 1 \leq i \leq n, \ 1 \leq j \leq m]$ be a polynomial ring over an arbitrary field $k$ and $M$ be an $n \times m$ matrix of variables in $R$ where $n\leq m$. Let $\Delta$ be an $(n-1)$-dimensional pure simplicial complex and assume that $<$ is an arbitrary term order on $R$. Assume that $\Delta^\vee$ satisfies the following:
\begin{enumerate}[(*)]
    \item for all $\sigma, \ \tau \in \Delta^\vee$, $|\sigma \cap \tau | < n-1$. 
\end{enumerate}
\end{setup}

\begin{cor}\label{cor:varDiffsRegular}
Adopt notation and hypotheses as in Setup \ref{set:rainDFIsetup}. Then the sequence of variable differences
    $$\{ x_{11} - x_{21} , \dots, x_{11} - x_{n1} \} \cup \cdots \cup \{ x_{1m} - x_{2m} , \dotsc , x_{1m} - x_{nm} \}$$
    forms a regular sequence on $R / \rain_< (J_\Delta)$ if and only if $\grade (\rain_< (J_\Delta) : g) = m-n$ for all $g \in \rain_< (J_{\Delta^\vee})$.
\end{cor}

\begin{proof}
Observe that $\grade (\rain_< (J_\Delta) : g) \geq m-n$ for all $g \in \rain_< (J_{\Delta^\vee})$ by Theorem \ref{thm:boocherBnos}. The grade after specializing must be $m-n$, so the variable differences are regular if and only if $\grade (\rain_< (J_\Delta) : g) = m-n$ for all $g \in \rain_< (J_{\Delta^\vee})$. The result then follows by Corollary \ref{cor:equiGenCor}.
\end{proof}

The next definition introduces the notion of a free sequence of vertices in a CW complex. This definition will end up being the ``correct" condition for characterizing rainbow DFIs with linear resolution.

\begin{definition}\label{def:freeSeq}
    Let $\cP$ be a CW complex. A vertex is \emph{free} if it is contained in a unique facet of $\cP$. A \emph{free sequence} is a sequence of vertices $\{ v_1 , \dots , v_k \} \subseteq \cP$ such that $v_i$ is a free vertex of $\cP \backslash \{ v_1 , \dots , v_{i-1} \}$ for each $1 = 1, \dots , k$. 
\end{definition}

\begin{lemma}[{\cite{NR09}}]\label{lem:NRlemma}
Let $C$ be a polytopal complex and $v$ a vertex in $C$ that lies in a unique facet $P$, with $\dim P > 0$. Then the subcomplex induced by deleting the vertex $v$ is homotopy equivalent to $C$. 
\end{lemma}

Given an $(n-1)$-dimensional pure simplicial complex $\Delta$, the facets of $\Delta$ have a one-to-one correspondence with vertices of $\po (E_\bullet^<)$. We will tacitly employ this association for the remainder of the paper. 

\begin{cor}\label{cor:freeSeqisLinear}
Adopt notation and hypotheses as in Corollary \ref{cor:initisCellular}. Let $\Delta$ be an $(n-1)$-dimensional pure simplicial complex. If the facets of $\Delta^\vee$ admit an ordering for which they are a free sequence in $\po (E_\bullet^<)$, then $\rain_< (J_\Delta)$ has linear minimal free resolution supported on a CW complex.
\end{cor}

\begin{proof}
The fact that the minimal free resolution is CW is Corollary \ref{cor:rainbowisCellular}. The linearity follows from Lemma \ref{lem:NRlemma}; since $\inn_< I_n (M)$ has linear resolution, removing a free vertex at each step leaves the homology unchanged. This implies that $\rain_< (J_\Delta)$ has linear resolution.
\end{proof}

\begin{theorem}\label{thm:linearResCriterion}
Adopt notation and hypotheses as in Setup \ref{set:rainDFIsetup}. Then $\rain_< (J_\Delta)$ has a linear minimal free resolution if and only if $\grade (\rain_< (J_\Delta) : g) = m-n$ for all $g \in G(\rain_< (J_{\Delta^\vee}) )$ (that is, each element of $\Delta^\vee$ is a free vertex of $\po (E_\bullet^<)$).
\end{theorem}

\begin{proof}
$\implies:$ Argue by contraposition. Recall that a free resolution of $R / \rain_< (J_\Delta)$ may be obtained by an iterated trimming complex taking the form of a mapping cone of a morphism of complexes:
\begin{equation*}
\xymatrix{\cdots \ar[r]^{d_{k+1}} &  F_{k} \ar[d]^{Q_{k-1}}\ar[r]^{d_{k}} & \cdots \ar[r]^{d_3} & F_2 \ar[rrrr]^{d_2'} \ar[d]^{Q_1} &&&& F_1' \ar[d]^{d_1} \\
\cdots \ar[r]^-{\bigoplus m^i_k} & \bigoplus_{i=1}^t K^i_{k-1} \ar[r]^-{\bigoplus m^i_{k-1}} & \cdots \ar[r]^-{\bigoplus m^i_2} & \bigoplus_{i=1}^t K^i_1 \ar[rrrr]^-{-\sum_{i=1}^t m^i_1(-)\cdot d_1(e^i_0)} &&&& R. \\}\end{equation*}
In the above, the complex on the top row is built from the minimal free resolution of $R / \inn_< I_n (M)$, which has length $m-n+1$. The bottom row is a direct sum of the minimal free resolutions of $(\rain_< (J_\Delta) : g)$ for each $g \in \rain_< (J_{\Delta^\vee})$; for complete details on this construction see \cite{vandebogert2020trimming}. The minimal free resolution will be linear if and only if the comparison maps $Q_i$ are surjective in each homological degree. If $\grade (\rain_< (J_\Delta) : g) > m-n$ for some $g \in \rain_< (J_{\Delta^\vee})$, however, then $Q_{m-n+1}$ will not be surjective because $F_{m-n+2} = 0$.

$\impliedby:$ The assumption $\grade (\rain_< (J_\Delta) : g) = m-n$ for all $g \in \rain_< (J_{\Delta^\vee})$ implies that each $g \in \Delta^\vee$ is a free vertex in $\po (E_\bullet^<)$. The conclusion follows from Corollary \ref{cor:freeSeqisLinear}.
\end{proof}

\begin{cor}\label{cor:Btab}
    Adopt notation and hypotheses as in Setup \ref{set:rainDFIsetup}. Then $R / \rain_< (J_\Delta)$ has Betti table
    \begin{equation*}\begin{tabular}{C|C C C C C C C C C}
     & 0 & 1 & \cdots &  \ell & \cdots & m-n+1  \\
     \hline 
   0  & 1 & \cdot & \cdots &  \cdot &\cdots &  \cdot \\
   
   \vdots & \cdot & \cdot & \cdots & \cdot & \cdots & \cdot \\
   
   n-1 & \cdot & \binom{m}{n}-r & \cdots & \binom{n+\ell-2}{\ell-1} \binom{m}{n+\ell-1} -r\binom{m-n}{\ell-1} & \cdots &  \binom{m-1}{m-n} - r .\\
\end{tabular}
\end{equation*}
\end{cor}

\begin{proof}
This follows from Corollary \ref{cor:varDiffsRegular} and Proposition \ref{prop:stdExercise} combined with the Betti tables given in \cite[Corollary 5.12]{vdb2020maxlideal}.
\end{proof}

\section{Applications to Polarizations}\label{sec:polarizationsandConc}

In this section, we consider applications of the machinery and results produced in the previous sections to polarizations of Artinian monomial ideals. We recall the result of Almousa, Fl\o ystad, and Lohne that shows Alexander duality provides an equivalence between rainbow monomial ideals and polarizations of Artinian monomial ideals. This immediately implies Proposition \ref{prop:polarizationProp} and Corollary \ref{cor:polarizationsCor}. Furthermore, Corollary \ref{cor:polarizationsCor} shows that any rainbow DFI $J_\Delta$ with $\Delta^\vee$ consisting only of free vertices is a polarization of a power of the maximal ideal if and only if $\rain_< J_\Delta = \inn_< I_n (M)$. We conclude with a variety of questions related to the results of this paper. 

We begin with the definition of an Alexander dual of a monomial ideal.

\begin{definition}
    Let $I$ be a monomial ideal. The Alexander dual $I^\vee$ is the ideal with generators consisting of all monomials with nontrivial common divisor with every generator of $I$.
    \end{definition}
    
    \begin{example}
    Let $I = \left({x}_{1}{x}_{2},{x}_{1}{x}_{3},{x}_{2}{x}_{3},{x}_{1}{x}_{4},{x}_{2}{
       x}_{4},{x}_{3}{x}_{4}\right)$. Then one computes:
       $$I^ \vee = \left({x}_{1}{x}_{2}{x}_{3},{x}_{1}{x}_{2}{x}_{4},{x}_{1}{x}_{3}{x
       }_{4},{x}_{2}{x}_{3}{x}_{4}\right).$$
    \end{example}
    
Next, we give the definition of a polarization. As previously mentioned, the intuition behind polarization is to replace an ideal with a squarefree ideal that is homologically indistinguishable. The existence of polarizations translates the problem of minimal free resolutions of arbitrary monomial ideals into the problem of minimal free resolutions of \emph{squarefree} monomial ideals.

\begin{definition}
    Let $I$ be an Artinian monomial ideal in the polynomial ring $S = k[x_1,\ldots, x_n]$, such that for every index $i$, $x_i^{m_i}$ is a minimal generator. Let $\check{X}_i = \{x_{i1},x_{i2},\ldots, x_{i m_i}\}$ be a set of variables, and let $\tilde S = k[\check{X}_1,\ldots \check{X}_n]$ be the polynomial ring in the union of these variables. An ideal $\tilde I\subset \tilde S$ is a \emph{polarization} of $I$ if
\begin{align*}
\sigma = \{x_{11}-x_{12}, x_{11}-x_{13},\ldots,x_{11}-x_{1m_1}\} \cup \{x_{21}-x_{22},\ldots,x_{21}-x_{2m_2}\}\cup \ldots \\
\cup \{x_{n1}-x_{n2},\ldots, x_{n1}-x_{n,m_n}\}
\end{align*}
is a regular sequence in $\tilde S / \tilde I$ and $\tilde I \otimes \tilde S / \sigma \cong I$.
    \end{definition}

\begin{prop}[{\cite{polarizations}}]\label{prop:polarizationProp}
    Let $J$ be an ideal generated by rainbow monomials with linear resolution (with all variables in the ambient ring occurring in some generator of $J$). Then $J^\vee$ is a polarization of an Artinian monomial ideal.
    \end{prop}
    
\begin{cor}
    Let $\Delta$ be an $(n-1)$-dimensional simplicial complex such that $\Delta^\vee$ admits an ordering for which it is a free sequence of $\po (E_\bullet^<)$. Then $(\rain_< (J_\Delta))^\vee$ is the polarization of an Artinian monomial ideal.
\end{cor}

\begin{proof}
Combine Proposition \ref{prop:polarizationProp} with Corollary \ref{cor:freeSeqisLinear}.
\end{proof}
    
\begin{cor}\label{cor:polarizationsCor}
    Adopt notation and hypotheses as in Setup \ref{set:rainDFIsetup}. Then: 
    \begin{enumerate}
        \item $(\rain_< (J_\Delta))^\vee$ is the polarization of an Artinian monomial ideal if and only if $\grade \big( \rain_< (J_\Delta) : g \big) = m-n$ for all $g \in G(\rain_< (J_{\Delta^\vee}))$, and
        \item in the situation of $(1)$, $(\rain_< (J_\Delta))^\vee$ is a polarization of a power of the graded maximal ideal if and only if $|\Delta^\vee| = 0$ (that is, $\rain_< (J_\Delta) = \inn_< I_n (M)$).
    \end{enumerate}
    \end{cor}

\begin{proof}
The statement of $(1)$ is immediate from Theorem \ref{thm:linearResCriterion}. To prove $(2)$, it suffices to show that if $|\Delta^\vee| >0$, then the monomial ideal $K$ with $G(\overline{K}) = \{ x_\aa \mid \aa \in \Delta^\vee \}$ is not Cohen-Macaulay. That is, $K^\vee$ does not have linear resolution. Let $x_\aa = x_{a_1} \cdots x_{a_n} \in G(\overline{K})$ and write $[m] \backslash \aa = \{ b_1 < \cdots < b_{m-n} \}$. It is clear that $x_{b_1} \cdots x_{b_{m-n}} \in K^\vee$. By Corollary \ref{cor:Btab}, $R/K^\vee$ has regularity $m-n+1$, whence $R/K^\vee$ does not have linear resolution.
\end{proof}

The following illustrates many of the results in this paper in a simple example.

\begin{example}
Let $<$ denote the standard diagonal term order, with $n=3$ and $m=5$. Assume that $\Delta^\vee = \{ (1,2,3) , (3,4,5) \}$. The associated cellular resolution is precisely the complex of boxes introduced by Nagel and Reiner. One can verify:
$$\rain_< (J_\Delta) : x_{11} x_{22} x_{33} = (x_{34} , x_{35}),   \quad \rain_< (J_\Delta) : x_{13} x_{24} x_{35} = (x_{11} , x_{12});$$
this shows that $(1,2,3)$ and $(3,4,5)$ are free vertices in the complex of boxes. As guaranteed by Corollary \ref{cor:Btab}, $R/\rain_< (J_\Delta)$ has Betti table
$$\begin{matrix}
       &0&1&2&3\\\text{total:}&1&8&11&4\\\text{0:}&1&\text{.}&\text{.}&\text{.}\\\text{1:}&\text
       {.}&\text{.}&\text{.}&\text{.}\\\text{2:}&\text{.}&8&11&4\\\end{matrix}$$
and one can verify using Macaulay2 \cite{M2} that $(\rain_< (J_\Delta))^\vee$ is a polarization of
$$(x_1^2 , x_1x_2^2 , x_2^3 , x_1 x_2 x_3 , x_2^2 x_3 , x_3^2).$$
\end{example}    

\begin{question}\label{question:euqivForFreeSeq}
Is Corollary \ref{cor:freeSeqisLinear} an equivalence? That is, if a rainbow DFI $\rain_< (J_\Delta)$ has linear resolution, then do the elements of $\Delta^\vee$ admit an ordering for which they are a free sequence in $\po (E_\bullet^<)$?
\end{question}

Question \ref{question:euqivForFreeSeq}, if answered in the affirmative, yields an interesting connection between polarizations and free sequences on the CW complex $\po (E_\bullet^<)$. Moreover, using the combinatorics of the initial ideals provided by Sturmfels and Zelevinsky \cite{SturmfelsZelevinsky93}, it may be possible to explicitly enumerate all free sequences on $\po (E_\bullet^<)$, thus constructing all possible polarizations with Alexander dual contained in $\inn_< I_n (M)$ for a given term order.

\begin{question}\label{question:sqfreeCellular}
Let $\Delta$ be an $(n-1)$-dimensional simplicial complex with the property that $|\sigma \cap \tau | < n-1$ for all $\sigma, \tau \in \Delta^\vee$. Does there exist a term order $<$ on $R$ such that every facet of $\Delta^\vee$ corresponds to a free vertex of $\po (E_\bullet^<)$?
\end{question}

A positive answer to Question \ref{question:sqfreeCellular} would immediately yield that a large class of equigenerated monomial ideals (considered in \cite{vdb2020maxlideal}) has CW resolution, since one can employ Corollary \ref{cor:varDiffsRegular} to specialize. More precisely, to the authors knowledge, it is not known if every equigenerated squarefree monomial ideal whose squarefree complement has no linear relations must have CW resolution.

\begin{question}\label{question:isitDG}
Let $<$ be any term order. Then does the sparse Eagon-Northcott complex admit the structure of an associative DG-algebra? If so, then how does this product compare to the product on the squarefree Eliahou-Kervaire resolution constructed by Peeva \cite{peeva19960} upon specialization?
\end{question}

The first part of Question \ref{question:isitDG} is likely not as difficult to answer as the second part. It would be interesting to see how multigraded DG-products on $E_\bullet^<$ specialize to give DG-products on the resolution of all squarefree monomial ideals of a given degree. It would be even \emph{more} interesting if these products were all distinct. If so, this would also yield a fascinating correspondence between choices of initial ideals of $\inn_< (M)$ and choices of algebra structures on the minimal free resolution of all squarefree monomials of degree $n$ in $m$ variables. Up to taking linear combinations, this may provide a method of parametrizing all such multigraded products.

\section*{Acknowledgements}

Thanks to Ayah Almousa for very helpful discussions related to this work. Thanks to Paolo Mantero for helpful comments and suggestions on an earlier draft of this paper.

\bibliographystyle{amsplain}
\bibliography{biblio}

\providecommand{\bysame}{\leavevmode\hbox to3em{\hrulefill}\thinspace}
\providecommand{\MR}{\relax\ifhmode\unskip\space\fi MR }
\providecommand{\MRhref}[2]{%
  \href{http://www.ams.org/mathscinet-getitem?mr=#1}{#2}
}
\providecommand{\href}[2]{#2}
\begin{thebibliography}{10}

\bibitem{polarizations}
Ayah Almousa, Gunnar Fl{\o}ystad, and Henning Lohne, \emph{Polarizations of
  powers of graded maximal ideals}, arXiv preprint arXiv:1912.03898 (2019).

\bibitem{batzies2002discrete}
Ekkehard Batzies and Volkmar Welker, \emph{Discrete morse theory for cellular
  resolutions}, Journal fur die Reine und Angewandte Mathematik \textbf{543}
  (2002), 147--168.

\bibitem{bayer1998monomial}
Dave Bayer, Irena Peeva, and Bernd Sturmfels, \emph{Monomial resolutions},
  Mathematical Research Letters \textbf{5} (1998), no.~1, 31--46.

\bibitem{bayer1998cellular}
Dave Bayer and Bernd Sturmfels, \emph{Cellular resolutions of monomial
  modules}, Journal f{\"u}r die reine und angewandte Mathematik \textbf{1998}
  (1998), no.~502, 123--140.

\bibitem{bjorner1984posets}
Anders Bj{\"o}rner, \emph{Posets, regular cw complexes and bruhat order},
  European Journal of Combinatorics \textbf{5} (1984), no.~1, 7--16.

\bibitem{bjorner1983lexicographically}
Anders Bj{\"o}rner and Michelle Wachs, \emph{On lexicographically shellable
  posets}, Transactions of the American Mathematical Society \textbf{277}
  (1983), no.~1, 323--341.

\bibitem{boocher2012}
Adam Boocher, \emph{Free resolutions and sparse determinantal ideals}, Math.
  Res. Lett \textbf{19} (2012), no.~04, 805--821.

\bibitem{BH98-CMrings}
Winfried Bruns and J{\"u}rgen Herzog, \emph{Cohen-macaulay rings}, no.~39,
  Cambridge university press, 1998.

\bibitem{clark2019minimal}
Timothy Clark and Alexandre Tchernev, \emph{Minimal free resolutions of
  monomial ideals and of toric rings are supported on posets}, Transactions of
  the American Mathematical Society \textbf{371} (2019), no.~6, 3995--4027.

\bibitem{craw2012cellular}
Alastair Craw and Alexander~Quintero V{\'e}lez, \emph{Cellular resolutions of
  noncommutative toric algebras from superpotentials}, Advances in Mathematics
  \textbf{229} (2012), no.~3, 1516--1554.

\bibitem{develin2007tropical}
Mike Develin and Josephine Yu, \emph{Tropical polytopes and cellular
  resolutions}, Experimental Mathematics \textbf{16} (2007), no.~3, 277--291.

\bibitem{DES98}
Persi Diaconis, David Eisenbud, and Bernd Sturmfels, \emph{Lattice walks and
  primary decomposition}, Mathematical Essays in Honor of Gian-Carlo Rota,
  Springer, 1998, pp.~173--193.

\bibitem{dochtermann2012cellular}
Anton Dochtermann and Alexander Engstr{\"o}m, \emph{Cellular resolutions of
  cointerval ideals}, Mathematische Zeitschrift \textbf{270} (2012), no.~1-2,
  145--163.

\bibitem{dochtermann2012tropical}
Anton Dochtermann, Michael Joswig, and Raman Sanyal, \emph{Tropical types and
  associated cellular resolutions}, Journal of Algebra \textbf{356} (2012),
  no.~1, 304--324.

\bibitem{dochtermann2014cellular}
Anton Dochtermann and Fatemeh Mohammadi, \emph{Cellular resolutions from
  mapping cones}, Journal of Combinatorial Theory, Series A \textbf{128}
  (2014), 180--206.

\bibitem{favacchio2018arithmetically}
Giuseppe Favacchio, Elena Guardo, and Juan Migliore, \emph{On the
  arithmetically cohen-macaulay property for sets of points in multiprojective
  spaces}, Proceedings of the American Mathematical Society \textbf{146}
  (2018), no.~7, 2811--2825.

\bibitem{favacchio2019multiprojective}
Giuseppe Favacchio and Juan Migliore, \emph{Multiprojective spaces and the
  arithmetically cohen--macaulay property}, Mathematical Proceedings of the
  Cambridge Philosophical Society, vol. 166, Cambridge University Press, 2019,
  pp.~583--597.

\bibitem{floystad2009cellular}
Gunnar Fl{\o}ystad, \emph{Cellular resolutions of cohen-macaulay monomial
  ideals}, Journal of Commutative Algebra \textbf{1} (2009), no.~1, 57--89.

\bibitem{M2}
Daniel~R. Grayson and Michael~E. Stillman, \emph{Macaulay2, a software system
  for research in algebraic geometry}, Available at
  \texttt{http://www.math.uiuc.edu/Macaulay2/}.

\bibitem{hatcher2005algebraic}
Allen Hatcher, \emph{Algebraic topology}, 2005.

\bibitem{herzog2010binomial}
J{\"u}rgen Herzog, Takayuki Hibi, Freyja Hreinsd{\'o}ttir, Thomas Kahle, and
  Johannes Rauh, \emph{Binomial edge ideals and conditional independence
  statements}, Advances in Applied Mathematics \textbf{45} (2010), no.~3,
  317--333.

\bibitem{herzog2015linear}
J{\"u}rgen Herzog, Dariush Kiani, and Sara~Saeedi Madani, \emph{The linear
  strand of determinantal facet ideals}, The Michigan Mathematical Journal
  \textbf{66} (2017), no.~1, 107--123.

\bibitem{madani16-BEIsurvey}
Sara~Saeedi Madani, \emph{Binomial edge ideals: A survey}, The 24th National
  School on Algebra, Springer, 2016, pp.~83--94.

\bibitem{mermin2010eliaiiou}
Jeffrey Mermin, \emph{The eliahou-kervaire resolution is cellular}, Journal of
  Commutative Algebra \textbf{2} (2010), no.~1, 55--78.

\bibitem{miller2004combinatorial}
Ezra Miller and Bernd Sturmfels, \emph{Combinatorial commutative algebra}, vol.
  227, Springer Science \& Business Media, 2004.

\bibitem{NR09}
Uwe Nagel and Victor Reiner, \emph{Betti numbers of monomial ideals and shifted
  skew shapes}, the electronic journal of combinatorics (2009), R3--R3.

\bibitem{ohtani2011}
Masahiro Ohtani, \emph{Graphs and ideals generated by some 2-minors},
  Communications in Algebra \textbf{39} (2011), no.~3, 905--917.

\bibitem{peeva19960}
Irena Peeva, \emph{0-borel fixed ideals}, Journal of Algebra \textbf{184}
  (1996), no.~3, 945--984.

\bibitem{peeva2010graded}
\bysame, \emph{Graded syzygies}, vol.~14, Springer Science \& Business Media,
  2010.

\bibitem{peeva2011frames}
Irena Peeva and Mauricio Velasco, \emph{Frames and degenerations of monomial
  resolutions}, Transactions of the American Mathematical Society (2011),
  2029--2046.

\bibitem{reiner2001linear}
Victor Reiner and Volkmar Welker, \emph{Linear syzygies of stanley-reisner
  ideals}, Mathematica Scandinavica (2001), 117--132.

\bibitem{sturmfels1996grobner}
Bernd Sturmfels, \emph{Grobner bases and convex polytopes}, vol.~8, American
  Mathematical Soc., 1996.

\bibitem{SturmfelsZelevinsky93}
Bernd Sturmfels and Andrei Zelevinsky, \emph{Maximal minors and their leading
  terms}, Advances in Mathematics \textbf{98} (1993), no.~1, 65--112.

\bibitem{vdb2020maxlideal}
Keller VandeBogert, \emph{Minimal free resolutions of certain equigenerated
  monomial ideals}, arXiv preprint arXiv:2007.02373 (2020).

\bibitem{vandebogert2020trimming}
\bysame, \emph{Trimming complexes and applications to resolutions of
  determinantal facet ideals}, Communications in Algebra (2020), 1--20.

\bibitem{velasco2008minimal}
Mauricio Velasco, \emph{Minimal free resolutions that are not supported by a
  cw-complex}, Journal of Algebra \textbf{319} (2008), no.~1, 102--114.

\end{thebibliography}
\addcontentsline{toc}{section}{Bibliography}

\end{document}